\documentclass[12pt, reqno]{amsart}
\usepackage{amsmath, amsthm, amscd, amsfonts, amssymb, graphicx, color}

 \input{mathrsfs.sty}
\newtheorem{theorem}{Theorem}[section]
\newtheorem{lemma}[theorem]{Lemma}

\theoremstyle{definition}
\newtheorem{definition}[theorem]{Definition}
\newtheorem{example}[theorem]{Example}

\theoremstyle{remark}
\newtheorem{remark}[theorem]{Remark}
\numberwithin{equation}{section}

\begin{document}

\title[Douglas factorization theorem revisited]{Douglas factorization theorem revisited}
\author[V. Manuilov, M.S. Moslehian, Q. Xu]{Vladimir Manuilov$^1$, M. S. Moslehian$^2$ \MakeLowercase{and} Qingxiang Xu$^3$}

\dedicatory{Dedicated to the memory of R. G. Douglas (1938-2018)}

\address{$^1$ Department of Mechanics and Mathematics, Moscow State University, Moscow, 119991, Russia.}
\email{manuilov@mech.math.msu.su}

\address{$^2$Department of Pure Mathematics, Ferdowsi University of Mashhad, P. O. Box 1159, Mashhad 91775, Iran.}
\email{moslehian@um.ac.ir}

\address{$^3$Department of Mathematics, Shanghai Normal University, Shanghai 200234, PR China.}
\email{qxxu@shnu.edu.cn; qingxiang\_xu@126.com}

%\author{Vladimir Manuilov$^1$}
%\address{$^1$Department of Mechanics and Mathematics, Moscow State University, Moscow, 119991, Russia}
%\email{manuilov@mech.math.msu.su}
\thanks{$^1$Partially supported by the RFBR grant No. 19-01-00574}

%    author two information
%\author{M. S. Moslehian$^2$}
%\address{$^2$Department of Pure Mathematics, Ferdowsi University of Mashhad, Center of Excellence in Analysis on ALgebraic %Structures (CEAAS), P. O. Box 1159, Mashhad 91775, Iran}
%\email{moslehian@um.ac.ir}
\thanks{$^2$Partially supported by a grant from Ferdowsi University of Mashhad (No. 2/50300)}

%    author three information
%\author{Qingxiang Xu$^3$}
%\address{$^3$Department of Mathematics, Shanghai Normal University, Shanghai 200234, PR China}
%\email{qxxu@shnu.edu.cn; qingxiang\_xu@126.com}
\thanks{$^3$Partially supported by a grant from Shanghai Municipal Science and Technology
Commission (18590745200)}


\begin{abstract}
Inspired by the Douglas factorization theorem, we investigate the solvability of the operator equation $AX=C$ in the framework of Hilbert $C^*$-modules. Utilizing partial isometries, we present its general solution when $A$ is a semi-regular operator. For such an operator $A$, we show that the equation $AX=C$ has a positive solution if and only if the range inclusion ${\mathcal R}(C) \subseteq {\mathcal R}(A)$ holds and $CC^*\le t\, CA^*$ for some $t>0$. In addition, we deal with the solvability of the operator equation $(P+Q)^{1/2}X=P$, where $P$ and $Q$ are projections. We provide a tricky counterexample to show that there exist a $C^*$-algebra $\mathfrak{A}$, a Hilbert $\mathfrak{A}$-module $\mathscr{H}$ and projections $P$ and $Q$ on $\mathscr{H}$ such that the operator equation $(P+Q)^{1/2}X=P$ has no solution. Moreover, we give a perturbation result related to the latter equation.
\end{abstract}

\maketitle

%---------------------------------------------------------------------------------------%
\section{Introduction}

The significant equation $AX=C$ and systems of equations including it have been intensely studied for matrices \cite{KM,RM}, bounded linear operators on Hilbert spaces \cite{ACG2,Dajic-Koliha, Liang-Deng}, and operators on Hilbert $C^*$-modules \cite{MEMM, WW}. For any operator $A$ between linear spaces, the range and the null space of $A$ are denoted by ${\mathcal R}(A)$ and ${\mathcal N}(A)$, respectively. In 1966, R. G. Douglas proved an equivalence of factorization, range inclusion, and majorization, known as the Douglas factorization theorem (Douglas lemma) in the literature. It reads as follows.

\begin{theorem}\label {DOUGLAS}\cite[Theorem 1]{DOU}
If $\mathcal{H}$ is a Hilbert space and $A, B \in \mathbb{B}(\mathcal{H})$, then the following statements are equivalent:
\begin{enumerate}
\item[{\rm (i)}] $\mathcal{R}(C) \subseteq \mathcal{R}(A)$;
\item[{\rm (ii)}] The equation $AX=C$ has a solution $X \in \mathbb{B}(\mathcal{H})$;
\item[{\rm (iii)}]$CC^* \leq k^2 AA^*$ for some $k\geq 0$.
\end{enumerate}
Moreover, if (i), (ii), and (iii) are valid, then there exists a unique operator $C$ (known as the Douglas Solution in the literature) so that
\begin{enumerate}
\item[{\rm (a)}] $\|X\|^2=\inf\{\mu|CC^*\leq \mu AA^*\}$;
\item[{\rm (b)}] $\mathcal{N}(C) = \mathcal{N}(X)$;
\item[{\rm (c)}] $\mathcal{R}(X) \subseteq \overline{\mathcal{R}(A^{*})}$.
\end{enumerate}
\end{theorem}

There are several applications of the Douglas factorization theorem in investigation of operator equations. For instance, Nakamoto \cite{NAK} studied the solvability of $XAX = B$ by employing the Douglas factorization theorem \ref{DOUGLAS}. In 2008, Arias, Corach, and Gonzalez \cite{ACG2} introduced the notion of reduced solution which is a generalization of the concept of Douglas solution. More precisely, let $A \in \mathbb{B}(\mathcal{H},\mathcal{K})$ and $C \in \mathbb{B}(\mathcal{G},\mathcal{K})$ be operators between Hilbert spaces such that $\mathcal{R}(C) \subseteq \mathcal{R}(A)$ and let $M$ be a closed subspace of $\mathcal{H}$ such that $\mathcal{N}(A) \oplus M= \mathcal{H}$. Then there exists a unique solution $X_M$ of the equation $AX_M =C$ such that ${\mathcal R}(X_M)\subseteq M$. The operator $X_M$ is called the \emph{reduced solution} of the equation $AX =C$ for the subspace $M$ in the framework of Hilbert spaces. They parametrized these solutions by employing generalized inverses.

Inner product $C^*$-modules are generalizations of inner product spaces by allowing inner products to take values in some $C^{*}$-algebras instead of the field of complex numbers. More precisely, an inner-product module over a $C^*$-algebra $\mathfrak{A}$ is a right $ \mathfrak{A}$-module equipped with an $ \mathfrak{A}$-valued inner product $\langle \cdot, \cdot \rangle: \mathscr{H} \times \mathscr{H} \to \mathfrak{A}$. If $\mathscr{H}$ is complete with respect to the induced norm defined by $\| x\| = \| \langle x,x\rangle \|^{\frac{1}{2}}\,\,(x\in \mathscr{H})$, then $\mathscr{H}$ is called a \emph{Hilbert $\mathfrak{A}$-module}.

Throughout the rest of this paper, $\mathfrak{A}$ denotes a $C^*$-algebra and $\mathscr{E}, \mathscr{H}, \mathscr{K}$, and $\mathscr{L}$ denote Hilbert $\mathfrak{A}$-modules. Let $\mathcal{L}(\mathscr{H},\mathscr{K})$ be the set of operators $A:\mathscr{H}\to \mathscr{K}$ for which there is an operator $A^*:\mathscr{K}\to
\mathscr{H}$ such that $\langle Ax,y\rangle=\langle x,A^*y\rangle$ for any $x\in \mathscr{H}$ and $y\in \mathscr{K}$. It is known that any element $A \in {\mathcal L}(\mathscr{H},\mathscr{K})$ must be bounded and $\mathfrak{A}$-linear. In general, a bounded operator between Hilbert $C^*$-modules may be not adjointable. We call ${\mathcal L}(\mathscr{H},\mathscr{K})$ the set of all Hermitian (adjointable) operators from $\mathscr{H}$ to $\mathscr{K}$. In the case when
$\mathscr{H}=\mathscr{K}$, $\mathcal{L}(\mathscr{H},\mathscr{H})$, abbreviated to $\mathcal{L}(\mathscr{H})$, is a
$C^*$-algebra. An operator $A\in \mathcal{L}(\mathscr{H})$ is positive if $\langle Ax,x\rangle\geq0$ for all $x\in{\mathcal H }$ (see \cite[Lemma~4.1]{Lance}), and we then write $A\geq 0$. For Hermitian operators $A, B\in\mathcal{L}(\mathscr{H})$, we say $B\geq A$ if $B-A \geq 0$. Let $\mathcal{L}(\mathscr{H})_{sa}$ and $\mathcal{L}(\mathscr{H})_+$ denote the set of Hermitian elements and positive elements in $\mathcal{L}(\mathscr{H})$, respectively.

A closed submodule $M$ of $\mathscr{H}$ is said to be \emph{orthogonally complemented} if $\mathscr{H}=M\oplus M^\perp$, where $M^\perp=\big\{x\in \mathscr{H}: \langle x,y\rangle=0\ \mbox{for any }\ y\in M\big\}$.
In this case, the projection from $\mathscr{H}$ onto $M$ is denoted by $P_M$. If $A\in \mathcal{L}( \mathscr{H}, \mathscr{K})$ does not have closed range, then neither $ {\mathcal N}(A) $ nor $ \overline{{\mathcal R}(A)} $ needs to be orthogonally complemented. In addition, if $A\in \mathcal{L}( \mathscr{H}, \mathscr{K})$ and $ \overline{{\mathcal R}(A^*)} $ is not orthogonally complemented, then it may happen that ${\mathcal N}(A)^{\bot}\neq \overline{{\mathcal R}(A^*)}$; see \cite{Lance, MT}. The above facts show that the theory Hilbert $C^*$-modules are much different and more complicated than that of Hilbert spaces.

There are several extensions of the Douglas factorization theorem in various settings; see \cite{FS, PS} as well as the recent survey \cite{MKX}. A generalization of the Douglas factorization theorem to the Hilbert $C^*$-module case was given as follows in which we do not need to assume that $\overline{{\mathcal R}(A^*)}$ is orthogonally complemented.\\

\begin{theorem} \cite[Corollary 2.5]{Fang-Moslehian-Xu} Let $\mathfrak{A}$ be a $C^*$-algebra, $\mathscr{E},\mathscr{H}$ and $\mathscr{K}$ be Hilbert $\mathfrak{A}$-modules. Let $A\in {\mathcal{L}}(\mathscr{E},\mathscr{K})$ and $A^\prime\in {\mathcal{L}}(\mathscr{H},\mathscr{K})$. Then the following statements are equivalent:
\begin{enumerate}
\item[{\rm (i)}] $A^\prime (A^\prime)^*\le \lambda AA^*$ for some $\lambda>0$;
\item[{\rm (ii)}]There exists $\mu>0$ such that $\Vert (A^\prime)^*z\Vert\le \mu \Vert A^*z\Vert$, for any $z\in \mathscr{K}$.\\
\end{enumerate}
\end{theorem}
In general, $A^\prime (A^\prime)^*\le \lambda AA^*$ for some $\lambda>0$ does not imply ${\mathcal R}(A^\prime)\subseteq {\mathcal R}(A) $. As an example, let $\mathcal{H}$ be a separable infinite dimensional Hilbert space, let $\mathfrak{A}=\mathscr{H}=\mathscr{K}=\mathbb{B}(\mathcal{H})$ and let $\mathscr{E}$ be the algebra $\mathbb{K}(\mathcal{H})$ of all compact operators. Suppose that $S={\rm diag}(1, 1/2, 1/3, \ldots)$ is the diagonal operator with respect to some orthonormal basis and define $A:\mathscr{E} \to \mathscr{K}$ by $A(T):= ST$ for $T\in \mathfrak{A}$, and set $A^\prime:= (AA^*)^{1/2}$.

We, however, have the following interesting result.
\begin{lemma}\label{lem:two orthogonally complemented conditions} {\rm (\cite[Theorem~3.2]{Fang-Moslehian-Xu} and \cite[Theorem 1.1]{Fang-Yu-Yao})} Let $A\in\mathcal{L}(\mathscr{H},\mathscr{K})$. Then the following statements are equivalent:
\begin{enumerate}
\item[{\rm (i)}] $\overline{{\mathcal R}(A^*)}$ is orthogonally complemented in $\mathscr{H}$;
\item[{\rm (ii)}] Let $C\in\mathcal{L}(\mathscr{L},\mathscr{K})$ be any such that ${\mathcal R}(C)\subseteq {\mathcal R}(A)$. Then the equation
\begin{equation}\label{equ:equation of AX=C}AX=C, X\in \mathcal{L}(\mathscr{L},\mathscr{H})\end{equation} has a reduced solution $D$, that is,
\begin{equation}\label{equ:property of the reduced solution}AD=C, D\in \mathcal{L}(\mathscr{L},\mathscr{H})\ \mbox{and}\ {\mathcal R}(D)\subseteq \overline{\mathcal{R}(A^*)}.\end{equation}
\end{enumerate}
\end{lemma}

It is remarkable that such a reduced solution (if it exists) is unique, and for Hilbert space operators as well as adjointable operators on Hilbert $C^*$-modules, most literatures on the solvability of equation \eqref{equ:equation of AX=C} are only focused on the regular case \cite{Dajic-Koliha, WW}, that is, the ranges of $A$ and the other associated operators are assumed to be closed. Very little has been done in the case when the associated operators are non-regular, which is the concern of this paper.

In view of the equivalence of Lemma~\ref{lem:two orthogonally complemented conditions}~(i) and (ii), the term of the semi-regularity for adjointable operators is introduced in this paper (see Definition~\ref{defn of semi-regularity}).
Such a semi-regularity condition is somehow natural in dealing with the solvability of equation \eqref{equ:equation of AX=C},
since it is always true for Hilbert space operators and if it fails to be satisfied, then equation \eqref{equ:equation of AX=C} may be unsolvable.
Furthermore, it is noted that for an adjointable operator $A$, $A$ is semi-regular if and only if $A$ has the polar decomposition $A=U|A|$ \cite[Proposition~15.3.7]{Wegge-Olsen}. So instead of the Moore-Penrose inverse in the regular case, one might use the partial isometry in the semi-regular case.
By utilizing partial isometries, we present the general solution of equation \eqref{equ:equation of AX=C} when $A$ is a semi-regular operator. For such an operator $A$,
the Hermitian solutions and the positive solutions of equation \eqref{equ:equation of AX=C} have been completely characterized in Section~\ref{sec:general solution to AX=C}
 of this paper; see Theorems~\ref{thm:Hermitian solution of AX=C}
and \ref{thm:technique result of positive solution of AX=C}. As a result, certain mistakes in \cite[Section~1]{Fang-Yu-Yao} are corrected
for adjointable operators on Hilbert $C^*$-modules, and
some generalizations of \cite[Section~3]{Liang-Deng} are obtained from the Hilbert space case to the Hilbert $C^*$-module case.

The shorted operators initiated in \cite{Anderson-Duffin} for Hermitian positive semi-definite matrices and generalized in \cite{Fillmore-Williams} for Hilbert space operators,
are closely related to the operator equation $(A+B)^\frac12 X=A^\frac12$, where $A$ and $B$ are two positive operators.
Such an operator equation is always solvable when the underlying spaces are Hilbert spaces. To show that the same is not true for adjointable operators on Hilbert $C^*$-modules, we focus on the special case that both $A$ and $B$ are projections.
In the last section of this paper, we provide a tricky counterexample to show that there exist a $C^*$-algebra $\mathfrak{A}$, a Hilbert $\mathfrak{A}$-module $\mathscr{H}$ and two projections $P$ and $Q$ on $\mathscr{H}$ such that the operator equation $(P+Q)^{1/2}X=P, X\in\mathcal{L}(\mathscr{H})$ has no solution. Moreover, given projections $P,Q\in\mathcal{L}(\mathscr{H})$, we show that for any $\varepsilon\in (0,1)$, there exists a projection $Q'\in\mathcal{L}(\mathscr{H})$ such that $\|Q-Q'\|<\varepsilon$ and the equation $(P+Q')^{1/2}X=P, X\in\mathcal{L}(\mathscr{H})$ has a solution.

\section{Solutions of the operator equation $AX=C$}\label{sec:general solution to AX=C}

We begin with the definition of the semi-regularity as follows:
\begin{definition}\label{defn of semi-regularity}An operator $A\in\mathcal{L}(\mathscr{H},\mathscr{K})$ is said to be \emph{semi-regular} if $\overline{\mathcal{R}(A)}$ and $\overline{\mathcal{R}(A^*)}$ are orthogonally complemented in $\mathscr{K}$ and $\mathscr{H}$, respectively.
\end{definition}

\begin{remark}Recall that $A\in\mathcal{L}(\mathscr{H},\mathscr{K})$ is said to be \emph{regular} if $\mathcal{R}(A)$ is closed. In this case, the Moore-Penrose inverse $A^\dag$ of $A$ exists. This is an operator $A^\dag$ such that $AA^\dag A=A$, $A^\dag AA^\dag=A^\dag$, and $AA^\dag=P_{\mathcal{R}(A)}$ and $A^\dag A=P_{\mathcal{R}(A^*)}$ are projections (\cite[Theorem~2.2]{Xu-Sheng}). Hence $A$ is semi-regular.
\end{remark}

\begin{lemma}\label{lem:Wegge-Olsen}{\rm\cite[Proposition~15.3.7]{Wegge-Olsen}}\ Let $A\in\mathcal{L}(\mathscr{H},\mathscr{K})$ be semi-regular. Then there exists a unique partial isometry $U_A\in\mathcal{L}(\mathscr{H},\mathscr{K})$ such that
\begin{equation}\label{equ:equations associated the polar decomposition}A=U_A(A^*A)^{\frac{1}{2}}\ \mbox{and}\ U_A^*U_A=P_{\overline{\mathcal{R}(A^*)}}.
\end{equation}
\end{lemma}

\begin{theorem}\label{thm:solution of AX=C} Let $A\in\mathcal{L}(\mathscr{H},\mathscr{K})$ be semi-regular.
Then for any $C\in\mathcal{L}(\mathscr{L},\mathscr{K})$, operator equation \eqref{equ:equation of AX=C}
has a solution if and only if $\mathcal{R}(C)\subseteq\mathcal{R}(A)$. In such case, the general solution of \eqref{equ:equation of AX=C} has the form
\begin{equation}\label{equ:form of solutions to AX=C}X=D+(I-U_A^*U_A)Y,\end{equation}
where $D\in\mathcal{L}(\mathscr{L},\mathscr{H})$ is the reduced solution of \eqref{equ:equation of AX=C}, $U_A\in\mathcal{L}(\mathscr{H},\mathscr{K})$ is the partial isometry satisfying \eqref{equ:equations associated the polar decomposition}, and $Y\in\mathcal{L}(\mathscr{L},\mathscr{H})$ is arbitrary.
\end{theorem}
\begin{proof} Suppose that $\mathcal{R}(C)\subseteq\mathcal{R}(A)$. By Lemma~\ref{lem:two orthogonally complemented conditions},
equation \eqref{equ:equation of AX=C} is solvable and its reduced solution $D$ satisfies \eqref{equ:property of the reduced solution}.
Since $U_A^*U_A=P_{\overline{\mathcal{R}(A^*)}}$ and $\overline{\mathcal{R}(A^*)}=\mathcal{N}(A)^\bot$, we have
\begin{align}\label{msm1}
\mathcal{N}(U_A)=\mathcal{N}(A)\ \mbox{and}\ A(I-U_A^*U_A)=0.
\end{align}
Therefore, any $X$ of the form \eqref{equ:form of solutions to AX=C} is a solution of equation \eqref{equ:equation of AX=C}.

On the other hand, given any solution $X$ of equation \eqref{equ:equation of AX=C}, we have
$$X-D\in\mathcal{N}(A)=\mathcal{N}(U_A)=\mathcal{N}(U_A^*U_A),$$
which leads to
$$X-D=(I-U_A^*U_A)(X-D),$$
hence $X$ has the form of \eqref{equ:form of solutions to AX=C} with $Y=X-D$ therein.
\end{proof}

\begin{remark} Suppose that $A\in\mathcal{L}(\mathscr{H},\mathscr{K})$ is regular and $C\in\mathcal{L}(\mathscr{L},\mathscr{K})$ is given such that $\mathcal{R}(C)\subseteq \mathcal{R}(A)$. Then $A^\dag C$ is the reduced solution of \eqref{equ:equation of AX=C}, so the general solution of \eqref{equ:equation of AX=C} has the form \eqref{equ:form of solutions to AX=C} with $D$ and $U_A^*U_A$ being replaced by $A^\dag C$ and $A^\dag A$, respectively.
\end{remark}

To study the Hermitian solutions of \eqref{equ:equation of AX=C}, we need the following lemmas.

\begin{lemma}\label{lem:rang characterization-1}{\rm \cite[Proposition~2.7]{Liu-Luo-Xu}} Let $A\in\mathcal{L}(\mathscr{H},\mathscr{K})$ and $B,C\in\mathcal{L}(\mathscr{E},\mathscr{H})$ be such that $\overline{\mathcal{R}(B)}=\overline{\mathcal{R}(C)}$. Then $\overline{\mathcal{R}(AB)}=\overline{\mathcal{R}(AC)}$.
\end{lemma}

\begin{lemma}\label{lem:equivalence of Hermitian and positive} Let $A,C\in\mathcal{L}(\mathscr{H},\mathscr{K})$ be such that $A$ is semi-regular and $\mathcal{R}(C)\subseteq \mathcal{R}(A)$. Let $D\in\mathcal{L}(\mathscr{H})$ be the reduced solution of \eqref{equ:equation of AX=C} with $\mathscr{L}=\mathscr{H}$ therein and $P=U_A^*U_A$, where $U_A\in\mathcal{L}(\mathscr{H},\mathscr{K})$ is the partial isometry satisfying \eqref{equ:equations associated the polar decomposition}. Then the following statements are valid:
\begin{enumerate}
\item[{\rm (i)}] $DP$ is Hermitian if and only if $CA^*$ is Hermitian;
\item[{\rm (ii)}] $DP$ is positive if and only if $CA^*$ is positive;
\item[{\rm (iii)}] If $CA^*$ is Hermitian and regular, then $DP$ is also regular.
\end{enumerate}
\end{lemma}
\begin{proof} By \eqref{equ:equations associated the polar decomposition} and \eqref{equ:property of the reduced solution}, we have
\begin{equation}\label{equ:property of the reduced solution+1}PD=D\ \mbox{and thus}\ D^*(I-P)=0,\end{equation}
which leads to
\begin{equation}\label{equ:verification of hermitian-1}\langle DPx,y\rangle=\langle DPx,Py\rangle\ \mbox{and}\ \langle PD^*x,y\rangle=\langle D^*Px,Py\rangle,\ \mbox{for any $x,y\in \mathscr{H}$.}
\end{equation}

(i) ``$\Longrightarrow$": If $DP$ is Hermitian (positive), then
$CA^*=(AD)(PA^*)=A(DP)A^*$ is also Hermitian (positive).

``$\Longleftarrow$": For any $u,v\in \mathscr{K}$,
\begin{align*}\langle D A^*u, A^*v\rangle&=\langle A^*u, D^*A^*v\rangle=\langle A^*u, C^*v\rangle=\langle CA^*u, v\rangle\\
&=\langle AC^*u, v\rangle=\langle C^*u, A^*v\rangle=\langle D^*A^*u, A^*v\rangle,
\end{align*}
which implies that
$$\langle DPx, Py\rangle=\langle D^*Px, Py\rangle,\quad {\rm ~for~all~} x,y\in \mathscr{H}.$$
The equation above together with \eqref{equ:verification of hermitian-1} yields $DP=(DP)^*$.

(ii) ``$\Longleftarrow$": For any $u\in \mathscr{K}$,
\begin{equation*}\langle D A^*u, A^*u\rangle=\langle A^*u, C^*u\rangle=\langle CA^*u, u\rangle\ge 0,
\end{equation*}
which gives, by \eqref{equ:verification of hermitian-1}, that
$$\langle DPx, x\rangle=\langle D Px, Px\rangle\ge 0, \ \mbox{for any $x\in \mathscr{H}$}.$$

(iii) By Lemma~\ref{lem:rang characterization-1}, we have
$$\mathcal{R}(CA^*)=\overline{\mathcal{R}(CA^*)}=\overline{\mathcal{R}(CP)}\supseteq \mathcal{R}(CP)\supseteq \mathcal{R}(CA^*),$$
hence
\begin{equation}\label{equ:ranges are equal--1} \overline{\mathcal{R}(CP)}=\mathcal{R}(CP)=\mathcal{R}(CA^*).\end{equation}

Given any $x\in\overline{\mathcal{R}(DP)}$, there exists a sequence $\{x_n\}$ in $\mathscr{H}$ such that $DPx_n\to x=Px$ (since $D=PD$).
Then
$$CPx_n=ADPx_n\to Ax=CA^*u=AC^*u\ \mbox{for some $u\in \mathscr{K}$ (see \eqref{equ:ranges are equal--1})}.$$
Hence, from \eqref{msm1}, we have $x-C^*u\in \mathcal{N}(A)=\mathcal{N}(U_A)$. Therefore, $Px=PC^*u$.
It follows that
$$x=Px=PC^*u=PD^*A^*u=DPA^*u\in\mathcal{R}(DP),$$
since $PD^*=DP$ by item (i) of this lemma (as $CA^*$ is Hermitian).
This completes the proof that $\overline{\mathcal{R}(DP)}=\mathcal{R}(DP)$.
\end{proof}

Now, we consider the Hermitian solutions of equation \eqref{equ:equation of AX=C}.
\begin{theorem}\label{thm:Hermitian solution of AX=C} Let $A\in\mathcal{L}(\mathscr{H},\mathscr{K})$ be semi-regular.
Then for any $C\in\mathcal{L}(\mathscr{H},\mathscr{K})$, the system
\begin{equation}\label{equ:Hermitian equation of AX=C}AX=C, X\in \mathcal{L}(\mathscr{H})_{sa}\end{equation}
has a solution if and only if
\begin{equation}\label{equ:Hermitian solvable conditions for AX=C}\mathcal{R}(C)\subseteq\mathcal{R}(A)\ \mbox{and}\ CA^*\ \mbox{is Hermitian}.\end{equation} In such case, the general solution of \eqref{equ:Hermitian equation of AX=C} has the form
\begin{equation}\label{equ:form of Hermitian solutions to AX=C}X=D+(I-U_A^*U_A)D^*+(I-U_A^*U_A)Y(I-U_A^*U_A),\end{equation}
where $D\in\mathcal{L}(\mathscr{H})$ is the reduced solution of \eqref{equ:equation of AX=C} with $\mathscr{L}=\mathscr{H}$ therein, $U_A\in\mathcal{L}(\mathscr{H},\mathscr{K})$ is the partial isometry satisfying \eqref{equ:equations associated the polar decomposition}, and $Y\in\mathcal{L}(\mathscr{H})_{sa}$ is arbitrary.
\end{theorem}
\begin{proof}Suppose that $X_0\in \mathcal{L}(\mathscr{H})_{sa}$ is a solution of \eqref{equ:Hermitian equation of AX=C}. Then by
Theorem~\ref{thm:solution of AX=C} we have $\mathcal{R}(C)\subseteq\mathcal{R}(A)$ and
\begin{equation}\label{equ:form of particalar solution to AX=C}X_0=D+(I-U_A^*U_A)Y_0\ \mbox{for some $Y_0\in\mathcal{L}(\mathscr{H})$},\end{equation}
which leads by $X_0^*=X_0$ to \begin{equation}\label{equ:expression of D-D star}D+(I-U_A^*U_A)Y_0=D^*+Y_0^*(I-U_A^*U_A).\end{equation}
Moreover, from \eqref{equ:property of the reduced solution+1} we have
$$(I-U_A^*U_A)(D-D^*)(I-U_A^*U_A)=0.$$
This together with \eqref{equ:expression of D-D star} yields $Z_0=Z_0^*$, where
\begin{equation}\label{equ:expression of Z 0}Z_0=(I-U_A^*U_A)Y_0^*(I-U_A^*U_A).\end{equation}
In view of \eqref{equ:property of the reduced solution+1}, \eqref{equ:expression of D-D star}, and \eqref{equ:expression of Z 0}, we have
\begin{align*}(I-U_A^*U_A)Y_0&=(I-U_A^*U_A)\big[D+(I-U_A^*U_A)Y_0\big]\\
&=(I-U_A^*U_A)\big[D^*+Y_0^*(I-U_A^*U_A)\big]\\
&=(I-U_A^*U_A)D^*+Z_0.
\end{align*}
Substituting the above into \eqref{equ:form of particalar solution to AX=C} yields
\begin{equation}\label{equ:form of particalar solution to AX=C+1}X_0=D+(I-U_A^*U_A)D^*+Z_0,\end{equation}
therefore $U_A^*U_AD^*=(D+D^*)+Z_0-X_0$, whence $U_A^*U_AD^*=DU_A^*U_A$ since both $X_0$ and $Z_0$ are Hermitian. Furthermore, it is clear from \eqref{equ:expression of Z 0}
that $$Z_0=(I-U_A^*U_A)Z_0(I-U_A^*U_A),$$ which indicates by \eqref{equ:form of particalar solution to AX=C+1} that
$X_0$ has the form \eqref{equ:form of Hermitian solutions to AX=C}.

Conversely, assume that \eqref{equ:Hermitian solvable conditions for AX=C} is fulfilled. Then by Lemma~\ref{lem:equivalence of Hermitian and positive} (i) $DU_A^*U_A$ is Hermitian, and it is easy to verify that any $X$ of the form \eqref{equ:form of Hermitian solutions to AX=C} is a solution of system \eqref{equ:Hermitian equation of AX=C}.
\end{proof}

\begin{remark} Let $A,C\in\mathcal{L}(\mathscr{H},\mathscr{K})$ be such that $A$ is regular and \eqref{equ:Hermitian solvable conditions for AX=C} is satisfied.
Unlike the assertion given in \cite[Theorem~1.2]{Fang-Yu-Yao}, the reduced solution $A^\dag C$ of \eqref{equ:equation of AX=C} may fail to be Hermitian. An interpretation can be given by using block matrices as follows:

Evidently, the operators $A,A^\dag$ and $C$ can be partitioned in the following way:
\begin{eqnarray*}A&=&\left.
 \begin{array}{c}
 \mathcal{R}(A) \\
 \mathcal{N}(A^*)\\
 \end{array}
 \right.\left(
 \begin{array}{cc}
 A_{11} & 0 \\
 0 & 0 \\
 \end{array}
 \right)\left.
 \begin{array}{c}
 \mathcal{R}(A^*)\\
 \mathcal{N}(A) \\
 \end{array},
 \right.\ \mbox{where $A_{11}$ is invertible},\\
A^\dag&=&\left.
 \begin{array}{c}
 \mathcal{R}(A^*) \\
 \mathcal{N}(A)\\
 \end{array}
 \right.\left(
 \begin{array}{cc}
 A_{11}^{-1} & 0 \\
 0 & 0 \\
 \end{array}
 \right)\left.
 \begin{array}{c}
 \mathcal{R}(A)\\
 \mathcal{N}(A^*) \\
 \end{array},
 \right.\\
C&=&\left.
 \begin{array}{c}
 \mathcal{R}(A) \\
 \mathcal{N}(A^*)\\
 \end{array}
 \right.\left(
 \begin{array}{cc}
 C_{11} & C_{12} \\
 C_{21} & C_{22} \\
 \end{array}
 \right)\left.
 \begin{array}{c}
 \mathcal{R}(A^*)\\
 \mathcal{N}(A) \\
 \end{array}.
 \right.\\
\end{eqnarray*}
Conditions of $AA^\dag C=C$ and $CA^*=(CA^*)^*$ can then be rephrased as
\begin{equation*}C_{21}=0, C_{22}=0\ \mbox{and}\ C_{11}A_{11}^*=A_{11}C_{11}^*,\end{equation*}
which gives the partitioned form of $A^\dag C$ as
$$A^\dag C=\left.
 \begin{array}{c}
 \mathcal{R}(A^*) \\
 \mathcal{N}(A)\\
 \end{array}
 \right.\left(
 \begin{array}{cc}
 A_{11}^{-1}C_{11} & A_{11}^{-1}C_{12} \\
 0 & 0 \\
 \end{array}
 \right)\left.
 \begin{array}{c}
 \mathcal{R}(A^*)\\
 \mathcal{N}(A) \\
 \end{array}.
 \right.$$
Clearly, $A^\dag C$ is Hermitian if and only if $C_{12}=0$.

In view of the observation above, a concrete counterexample to \cite[Theorem~1.2]{Fang-Yu-Yao} can be constructed as follows:
\begin{example}\label{equ:counterexample-1} Let $A=\left(
 \begin{array}{cc}
 1 & 0 \\
 0 & 0 \\
 \end{array}
 \right)$, $C=\left(
 \begin{array}{cc}
 2 & 1 \\
 0 & 0 \\
 \end{array}
 \right)$, $X=\left(
 \begin{array}{cc}
 2 & 1 \\
 1 & 1 \\
 \end{array}
 \right)$, $Y=\left(
 \begin{array}{cc}
 1 & 1 \\
 1 & 0 \\
 \end{array}
 \right)$. Then $AX=C$ and $X=Y^*Y\ge 0$, whereas $A^\dag C\ne (A^\dag C)^*$.
\end{example}
\end{remark}

To study the positive solutions of equation \eqref{equ:equation of AX=C}, we need the following lemma.
\begin{lemma}\label{lem:2 by 2 positive operator matrix}{\rm \cite[Corollary 3.5]{Xu-Sheng}} Let $A=\left( \begin{array}{cc}
 A_{11} & A_{12}\\
 A_{12}^* & A_{22}\\
 \end{array} \right)\in\mathcal{L}(\mathscr{H}\oplus \mathscr{K})$ be Hermitian, where $A_{11}\in \mathcal{L}(\mathscr{H})$ is regular. Then $A\ge 0$ if and only if
\begin{enumerate}
\item[{\rm (i)}] $A_{11}\geq 0$ ;
\item[{\rm (ii)}] $A_{12}= A_{11} A_{11}^\dag A_{12}$;
\item[{\rm (iii)}] $A_{22}-A_{12}^*A_{11}^\dag A_{12}\geq 0$.
\end{enumerate}
\end{lemma}

Our technical result on the positive solutions of \eqref{equ:equation of AX=C} is as follows:
\begin{lemma}\label{lem:technique result of positive solution of AX=C} Let $A,C\in\mathcal{L}(\mathscr{H},\mathscr{K})$ be such that $A$ is semi-regular.
Then the system
\begin{equation}\label{equ:positive equation of AX=C}AX=C, X\in \mathcal{L}(\mathscr{H})_+\end{equation}
 has a solution if and only if
 \begin{equation}\label{equ:defn of lambda}\mathcal{R}(C)\subseteq \mathcal{R}(A), CA^*\in\mathcal{L}(\mathscr{H})_+\ \mbox{and}\ \lambda=\sup\big\{\Vert T_n\Vert: n\in\mathbb{N}\big\}<+\infty,\end{equation}
where
\begin{equation}\label{equ:defn of T n}T_n=(I-U_A^*U_A)D^*\Big[\frac{1}{n}I_{\mathscr{H}_1}+DU_A^*U_A|_{\mathscr{H}_1}\Big]^{-1}D(I-U_A^*U_A),\ \mbox{for each $n\in \mathbb{N}$},\end{equation}
in which $U_A\in\mathcal{L}(\mathscr{H},\mathscr{K})$ is the partial isometry satisfying \eqref{equ:equations associated the polar decomposition}, $\mathscr{H}_1=U_A^*U_A \mathscr{H}$ and $D\in\mathcal{L}(\mathscr{H})$ is the reduced solution of \eqref{equ:equation of AX=C} with $\mathscr{L}=\mathscr{H}$ therein.
If \eqref{equ:defn of lambda} is fulfilled, then the general solution of \eqref{equ:positive equation of AX=C} has the form \eqref{equ:form of Hermitian solutions to AX=C} with $Y\in\mathcal{L}(\mathscr{H})_+$ therein such that
\begin{equation}\label{equ:general positive solution-2}(I-U_A^*U_A)Y(I-U_A^*U_A)\ge T_n\ \mbox{for all $n \in \mathbb{N}$}.\end{equation}
\end{lemma}
\begin{proof} For simplicity, we put
\begin{equation}\label{equ:defn of P and H1}P=U_A^*U_A\ \mbox{and thus}\ \ \mathscr{H}_1=P\mathscr{H}.\end{equation} Suppose that $X\in\mathcal{L}(\mathscr{H})_+$ is a solution of system \eqref{equ:positive equation of AX=C}. Then \eqref{equ:Hermitian solvable conditions for AX=C} is fulfilled and $X$ has form \eqref{equ:form of Hermitian solutions to AX=C}. Therefore, by \eqref{equ:property of the reduced solution+1} and \eqref{equ:form of Hermitian solutions to AX=C},
$$DP=PXP\ge 0\ \mbox{and thus}\ CA^*\ge 0\ \mbox{by Lemma~\ref{lem:equivalence of Hermitian and positive} (ii)}.$$
For each $n\in\mathbb{N}$, let
$X_n=X+\frac{1}{n}P\in\mathcal{L}(\mathscr{H})_+$. Then from \eqref{equ:property of the reduced solution+1}, we get

\begin{equation*}X_n=\left.
 \begin{array}{c}
 \mathscr{H}_1 \\
 \mathscr{H}_1^\bot \\
 \end{array}
 \right.\left(
 \begin{array}{cc}
 \frac{1}{n}I_{\mathscr{H}_1}+DP|_{\mathscr{H}_1} & D(I-P)|_{\mathscr{H}_1^\bot} \\
 (I-P)D^*|_{\mathscr{H}_1} & (I-P)X(I-P)|_{\mathscr{H}_1^\bot}\\
 \end{array}
 \right)\left.
 \begin{array}{c}
 \mathscr{H}_1 \\
 \mathscr{H}_1^\bot \\
 \end{array}
 \right..
\end{equation*}
A direct application of Lemma~\ref{lem:2 by 2 positive operator matrix} to the operator $X_n$ above yields
\begin{equation*}0\le T_n\le (I-P)X(I-P)\ \mbox{for all $n\in\mathbb{N}$}.
\end{equation*}
hence
\begin{equation*}\lambda=\sup\big\{\Vert T_n\Vert: n\in\mathbb{N}\big\}\le \Vert (I-P)X(I-P) \Vert<+\infty.
\end{equation*}

Conversely, suppose that \eqref{equ:defn of lambda} is fulfilled. Then by Lemma~\ref{lem:equivalence of Hermitian and positive}~(ii) $DP$ is positive.
For each $n\in\mathbb{N}$, let
\begin{equation*}Z_n=\frac{1}{n}P+D+(I-P)D^*+\lambda (I-P).\end{equation*}
Then $Z_n$ is positive by Lemma~\ref{lem:2 by 2 positive operator matrix}, since it has the partitioned form
\begin{equation*}Z_n=\left.
 \begin{array}{c}
 \mathscr{H}_1 \\
 \mathscr{H}_1^\bot \\
 \end{array}
 \right.\left(
 \begin{array}{cc}
 \frac{1}{n}I_{\mathscr{H}_1}+DP|_{\mathscr{H}_1} & D(I-P)|_{\mathscr{H}_1^\bot} \\
 (I-P)D^*|_{\mathscr{H}_1} & \lambda I_{\mathscr{H}_1^\bot}\\
 \end{array}
 \right)\left.
 \begin{array}{c}
 \mathscr{H}_1 \\
 \mathscr{H}_1^\bot \\
 \end{array}
 \right.
\end{equation*}
and
\begin{equation*}\lambda (I-P)-T_n\ge \lambda (I-P)-\Vert T_n\Vert (I-P)=(\lambda-\Vert T_n\Vert)(I-P)\ge 0.\end{equation*}

Let $$X=\lim\limits_{n\to\infty}Z_n=D+(I-P)D^*+\lambda (I-P).$$ Then $X$ is positive and $AX=C$.

Finally, suppose that \eqref{equ:defn of lambda} is satisfied. Given any $X\in\mathcal{L}(\mathscr{H})_{sa}$ of form~\eqref{equ:form of Hermitian solutions to AX=C},
it is clear that $X\ge 0$ if and only if $X+\frac{1}{n}P\ge 0$ for any $n\in\mathbb{N}$. Based on such an observation and the direct application of Lemma~\ref{lem:2 by 2 positive operator matrix}, the asserted form of the general solution of \eqref{equ:positive equation of AX=C} follows.
\end{proof}

\begin{remark} In the preceding lemma, there is no
regularity or semi-regularity assumption on $DU_A^*U_A$. It is interesting to determine conditions under which the number $\lambda$ defined by \eqref{equ:defn of lambda} is finite. With the notations and the conditions of Lemma~\ref{lem:technique result of positive solution of AX=C} (except for $\lambda<+\infty$), for each $n\in\mathbb{N}$ let
\begin{equation}\label{equ:defn of S n}S_n=\frac{1}{n}I_{\mathscr{H}_1}+DU_A^*U_A|_{\mathscr{H}_1}\in\mathcal{L}(\mathscr{H}_1).\end{equation}
Note that from the assumption and Lemma~\ref{lem:equivalence of Hermitian and positive} we conclude that $DU_A^*U_A$ is positive, hence $U_A^*U_AD^*=DU_A^*U_A$.
Then
$$T_n=(I-U_A^*U_A)D^*S_n^{-\frac12} \left((I-U_A^*U_A)D^*S_n^{-\frac12}\right)^*,$$ hence
\begin{eqnarray*}\Vert T_n\Vert=\Big\Vert S_n^{-\frac12} D(I-U_A^*U_A)D^*S_n^{-\frac12}\Big\Vert=\Vert U_n+V_n\Vert,
\end{eqnarray*}
where
\begin{equation}\label{equ:defn of U n and V n}U_n=S_n^{-\frac12}(DU_A^*U_A)^2S_n^{-\frac12}\ \mbox{and}\ V_n=S_n^{-\frac12}DD^*S_n^{-\frac12}
\end{equation}
are such that $\Vert U_n\Vert\le \Vert DU_A^*U_A\Vert$ for all $n\in\mathbb{N}$, and $\Vert V_n\Vert=\big\Vert D^*S_n^{-1}D\big\Vert$ for each $n\in\mathbb{N}$.
Therefore,
\begin{align}\label{eqn:defn of mu}\lambda<+\infty&\Longleftrightarrow \sup\left\{\Vert V_n\Vert: n\in\mathbb{N}\right\}<+\infty\\
\label{eqn:defn of mu+1}&\Longleftrightarrow\sup\left\{\big\Vert D^*S_n^{-1}D\big\Vert: n\in\mathbb{N}\right\}<+\infty.\end{align}
\end{remark}

Based on the observation above, an application of Lemma~\ref{lem:technique result of positive solution of AX=C} is as follows:
\begin{theorem}\label{thm:technique result of positive solution of AX=C}{\rm (cf.\,\cite[Theorem~3.1 (iii)]{Liang-Deng})} Let $A,C\in\mathcal{L}(\mathscr{H},\mathscr{K})$ be such that $A$ is semi-regular.
Then system \eqref{equ:positive equation of AX=C}
 has a solution if and only if
 \begin{equation}\label{equ:defn of lambda-new}\mathcal{R}(C)\subseteq \mathcal{R}(A)\ \mbox{and}\ CC^*\le t\, CA^*\ \mbox{for some $t>0$}.\end{equation}
\end{theorem}
\begin{proof}Suppose that $X\in\mathcal{L}(\mathscr{H})_+$ is such that $AX=C$. Then $\mathcal{R}(C)\subseteq \mathcal{R}(A)$ and
$$CC^*=AX^2A^*\le \Vert X\Vert\, AXA^*=\Vert X\Vert\,CA^*.$$
Therefore, \eqref{equ:defn of lambda-new} is satisfied for any $t\ge \Vert X\Vert$.

Conversely, suppose that \eqref{equ:defn of lambda-new} is satisfied. Let $P$, $\mathscr{H}_1$, $S_n$ and $V_n$ be defined by \eqref{equ:defn of P and H1},
\eqref{equ:defn of S n}, and \eqref{equ:defn of U n and V n}, respectively. Then $DP$ is positive by Lemma~\ref{lem:equivalence of Hermitian and positive}~(ii), and from the latter condition in \eqref{equ:defn of lambda-new} we have
\begin{eqnarray*}\langle DD^* A^*x, A^*x\rangle=\langle ADD^* A^*x, x\rangle=\langle CC^*x, x\rangle\le t \langle CA^*x, x\rangle=t\langle D A^*x, A^*x\rangle
\end{eqnarray*}
for any $x\in \mathscr{K}$. Thus
$$\langle DD^*Pu, u\rangle=\langle DD^*Pu, Pu\rangle\le t\langle D Pu, Pu\rangle=t\langle DPu, u\rangle,$$
whence $DD^*P\le tDP$.
Accordingly,
$$\Vert V_n\Vert=\|S_n^{-\frac12}DD^*P S_n^{-\frac12}\|\le t\Big\Vert S_n^{-\frac12}\cdot DP \cdot S_n^{-\frac12}\Big\Vert\le t,\ \mbox{for any $n\in\mathbb{N}$}.$$
The conclusion then follows from \eqref{eqn:defn of mu} and Lemma~\ref{lem:technique result of positive solution of AX=C}.
\end{proof}

\begin{remark}Let $S_n$ be defined by \eqref{equ:defn of S n}, where $P=U_A^*U_A$ and $DP$ is positive.
Obviously, a sufficient condition for $\lambda<+\infty$ can be derived from \eqref{eqn:defn of mu+1} as
\begin{equation}\label{equ:defn of M}M=\sup\{\Vert S_n^{-1}\Vert: n\in\mathbb{N}\}<+\infty.\end{equation}
In this case, for any $n\in\mathbb{N}$ and $x\in \mathscr{H}_1$ we have
$$\Vert x\Vert \le \Vert S_n^{-1}\Vert\,\Vert S_n(x)\Vert \le M \Vert S_n(x)\Vert,$$
which leads to
$$\Vert DP x\Vert=\lim_{n\to\infty}\|S_n(x)\|\ge \frac{1}{M+1}\Vert x\Vert,\ \mbox{for any $x\in \mathscr{H}_1$}.$$
Therefore, $DP|_{\mathscr{H}_1}$ and furthermore $DP$
is regular, since $DP=PDP$.

\end{remark}
Our next result on the positive solutions of \eqref{equ:equation of AX=C} is as follows:
\begin{theorem}\label{thm:positive solution of AX=C-regular case} Let $A,C\in\mathcal{L}(\mathscr{H},\mathscr{K})$ be such that
$A$ is semi-regular and $CA^*$ is regular.
Then system \eqref{equ:positive equation of AX=C}
has a solution if and only if
\begin{equation}\label{equ:delete lambda}\mathcal{R}(C)\subseteq \mathcal{R}(A), CA^*\in\mathcal{L}(\mathscr{H})_+\ \mbox{and}\ \mathcal{R}(D)= \mathcal{R}(DP),\end{equation}
where $U_A\in\mathcal{L}(\mathscr{H},\mathscr{K})$ is the partial isometry satisfying \eqref{equ:equations associated the polar decomposition}, $P=U_A^*U_A$ and $D\in\mathcal{L}(\mathscr{H})$ is the reduced solution of \eqref{equ:equation of AX=C} with $\mathscr{L}=\mathscr{H}$ therein.
In such case, the general solution of \eqref{equ:positive equation of AX=C} has the form
\begin{equation}\label{equ:form of positive solutions to AX=C-regular case}X=X_0+PZP,\end{equation}
in which $Z\in\mathcal{L}(\mathscr{H})_+$ is arbitrary, and
\begin{equation}\label{equ:particular positive solution of X0-1}X_0=D+(I-P)D^*+(I-P)D^*(DP)^\dag D(I-P).\end{equation}
\end{theorem}
\begin{proof} Let $\mathscr{H}_1$ be defined by \eqref{equ:defn of P and H1}. Suppose that $X\in\mathcal{L}(\mathscr{H})_+$ is a solution of system \eqref{equ:positive equation of AX=C}. Then the first two conditions in \eqref{equ:delete lambda} is satisfied by Lemma~\ref{lem:technique result of positive solution of AX=C}; $DP$ is positive and regular by Lemma~\ref{lem:equivalence of Hermitian and positive}~(ii) and (iii); and by Theorem~\ref{thm:Hermitian solution of AX=C} there exists $Y\in\mathcal{L}(\mathscr{H})_{sa}$ such that $X$ has form \eqref{equ:form of Hermitian solutions to AX=C}, which leads to
\begin{equation}\label{equ:partitioned form of Hermitian solution to AX=C}X=\left.
 \begin{array}{c}
 \mathscr{H}_1 \\
 \mathscr{H}_1^\bot \\
 \end{array}
 \right.\left(
 \begin{array}{cc}
 DP|_{\mathscr{H}_1} & D(I-P)|_{\mathscr{H}_1^\bot} \\
 (I-P)D^*|_{\mathscr{H}_1} & (I-P)Y(I-P)|_{\mathscr{H}_1^\bot}\\
 \end{array}
 \right)\left.
 \begin{array}{c}
 \mathscr{H}_1 \\
 \mathscr{H}_1^\bot \\
 \end{array}
 \right.
\end{equation}
In view of \eqref{equ:partitioned form of Hermitian solution to AX=C} and the regularity together with the positivity of $DP$, we conclude from Lemma~\ref{lem:2 by 2 positive operator matrix} that 
\begin{equation}\label{equ:concerning range of D}\mathcal{R}\big(D(I-P)|_{\mathscr{H}_1^\bot}\big)\subseteq \mathcal{R}\big(DP|_{\mathscr{H}_1}\big)\end{equation} and
\begin{equation}\label{equ:defn of positive Z}Z\stackrel{def}{=}(I-P)Y(I-P)-(I-P)D^*(DP)^\dag D(I-P)\ge 0.\end{equation}
Formula~\eqref{equ:form of positive solutions to AX=C-regular case} for $X$ then follows from \eqref{equ:form of Hermitian solutions to AX=C}, \eqref{equ:particular positive solution of X0-1}, and \eqref{equ:defn of positive Z}, since it is obvious that $(I-P)Z(I-P)=Z$.
Furthermore, it is clear that \eqref{equ:concerning range of D} is satisfied if and only if 
$\mathcal{R}\big(D(I-P)\big)\subseteq \mathcal{R}(DP)$, which can obviously be rephrased as $\mathcal{R}(D)=\mathcal{R}(DP)$.

The discussion above indicates that when \eqref{equ:delete lambda} is satisfied, any $X\in \mathcal{L}(\mathscr{H})_+$ is a solution of system \eqref{equ:positive equation of AX=C}
if and only if it has the form \eqref{equ:form of positive solutions to AX=C-regular case}.
\end{proof}

\begin{remark} Let $A,C\in\mathcal{L}(\mathscr{H},\mathscr{K})$ be such that
$A$ is semi-regular, $\mathcal{R}(C)\subseteq \mathcal{R}(A)$ and $DU_A^*U_A$ is regular. Then from the proof of Theorem~\ref{thm:positive solution of AX=C-regular case} we can conclude that system \eqref{equ:positive equation of AX=C} has a solution if and only if $DU_A^*U_A$ is positive and 
$\mathcal{R}(D)=\mathcal{R}(DU_A^*U_A)$. In such case, the general solution of \eqref{equ:positive equation of AX=C} also has form
\eqref{equ:form of positive solutions to AX=C-regular case}.

It is noticeable that $CA^*$ may be non-regular even if $DU_A^*U_A$ is positive and regular. For example, let $A$ be semi-regular and meanwhile be non-regular, and put $C=A$. Then clearly, $U_A^*U_A$ is the reduced solution of $AX=A$. It is known that $A$ is regular if and only if $AA^*$ is regular (see \cite[Theorem 3.2]{Lance} and \cite[Remark~1.1]{Xu-Sheng}), so in this case $CA^*$ fails to be regular.
\end{remark}

\begin{remark}\label{rem:Zhangs observation} The conclusion stated in Theorem~\ref{thm:positive solution of AX=C-regular case} may be false if the last condition in \eqref{equ:delete lambda} is not fulfilled. For example, let $\mathfrak{A}=\mathbb{C}, \mathscr{H}=\mathscr{K}=\mathbb{C}^3$ and Put
$$A=\left(
 \begin{array}{ccc}
 1 & 0 & 0 \\
 0 & 1 & 0 \\
 0 & 0 & 0 \\
 \end{array}
 \right), C=\left(
 \begin{array}{ccc}
 1 & 0 & 0 \\
 0 & 0 & 1 \\
 0 & 0 & 0 \\
 \end{array}
 \right)\in\mathcal{L}(\mathscr{H}).$$
Then $\mathcal{R}(C)=\mathcal{R}(A), D=A^\dag C=C$, $P=U_A^*U_A=A$ and $CA^*=DP={\rm diag}(1,0,0)\in\mathcal{L}(\mathscr{H})_+$. 
Therefore, $\mathcal{R}(D)\ne \mathcal{R}(DP)$. Let $X=(x_{ij})_{1\le i,j\le 3}$ be Hermitian such that $AX=C$. Then direct computation yields 
$$X=\left(
 \begin{array}{ccc}
 1 & 0 & 0 \\
 0 & 0 & 1 \\
 0 & 1 & x_{33} \\
 \end{array}
 \right),$$
which can never be positive for any $x_{33}\in\mathbb{C}$.
\end{remark}

\begin{remark}Example~\ref{equ:counterexample-1} also indicates the wrong assertion in \cite[Theorem~1.3]{Fang-Yu-Yao}.
\end{remark}

\begin{remark} Based on quite different methods from ours, the solvability, Hermitian solvability and positive solvability of operator equation \eqref{equ:equation of AX=C} were considered recently in \cite[Section~3]{Liang-Deng} for Hilbert space operators. With the restriction of the regularities of the Hilbert space operators $A$ and $CA^*$, the positive solvability of
equation \eqref{equ:equation of AX=C} was considered in \cite[Theorem~5.2]{Dajic-Koliha}. The real positive solvability of
equation \eqref{equ:equation of AX=C} was studied in \cite[Section~4]{Liang-Deng} for Hilbert space operators. The latter topic can also be dealt with by following the line in the proof of Lemma~\ref{lem:technique result of positive solution of AX=C}.
\end{remark}

\section{Solvability of $AX=C$ associated with projections}

In this section, we study the solvability of the following operator equation
\begin{equation}\label{equ:operator equation-no solution}(P+Q)^{1/2}X=P, \ X\in \mathcal{L}(\mathscr{H}),\end{equation}
where $P,Q\in\mathcal{L}(\mathscr{H})$ are projections.

\begin{theorem} There exist a $C^*$-algebra $\mathfrak{A}$, a Hilbert $C^*$-module $\mathscr{H}$ over $\mathfrak{A}$ and two projections $P$ and $Q$ in $\mathcal{L}(\mathscr{H})$ such that operator equation \eqref{equ:operator equation-no solution} has no solution.
\end{theorem}

\begin{proof} Let $M_2(\mathbb{C})$ be the set of $2\times 2$ complex matrices, and $\mathfrak{B}=C\big([0,1];M_2(\mathbb{C})\big)$ be the set of continuous matrix-valued functions from $[0,1]$ to $M_2(\mathbb{C})$. Put
\begin{equation}\label{A}
\mathfrak{A}=\{f\in \mathfrak{B}:f(0)\ \mbox{and}\ f(1)\mbox{\ are both\ diagonal}\},
\end{equation}
and $\mathscr{H}=\mathfrak{A}$. With the inner product given by
$$\langle x,y\rangle=x^*y \ \mbox{for any $x,y\in \mathscr{H}$},$$ $\mathscr{H}$ becomes a Hilbert $\mathfrak{A}$-module such that $\mathcal{L}(\mathscr{H})=\mathfrak{A}$.

For shortness' sake, set
$$c_t=\cos\frac{\pi}{2}t\ \mbox{and}\ s_t=\sin\frac{\pi}{2}t,\ \mbox{for each $t\in[0,1]$}.$$ The matrix-valued functions $P(t)=\left(\begin{matrix}1&0\\0&0\end{matrix}\right)$, $Q(t)=\left(\begin{matrix}c_t^2&s_tc_t\\s_tc_t&s_t^2\end{matrix}\right)$ determine projections $P_{\mathfrak{A}}$ and $Q_{\mathfrak{A}}$, respectively, in $\mathfrak{A}$.

Note that $P(t)+Q(t)$ is invertible for all $t\in (0,1]$ (and not invertible for $t=0$). Indeed, $\left|\begin{matrix}1+c_t^2&s_tc_t\\s_tc_t&s_t^2\end{matrix}\right|=(1+c_t^2)s_t^2-s_t^2c_t^2=s_t^2$. Standard calculation shows that
 $$
(P(t)+Q(t))^{1/2}=\left(\begin{matrix}\alpha(t)&\beta(t)\\\beta(t)&\gamma(t)\end{matrix}\right),
 $$
where
\begin{eqnarray*}
\alpha(t)&=&\frac{1}{2}(2-s_t)(\sqrt{1+c_t}+\sqrt{1-c_t}),\\
\beta(t)&=&\frac{1}{2}s_t(\sqrt{1+c_t}-\sqrt{1-c_t}),\\
\gamma(t)&=&\frac{1}{2}s_t(\sqrt{1+c_t}+\sqrt{1-c_t}),
\end{eqnarray*}
hence
$$
(P(t)+Q(t))^{-1/2}=\frac{1}{s_t}\left(\begin{matrix}\gamma(t)&-\beta(t)\\-\beta(t)&\alpha(t)\end{matrix}\right), \ \mbox{for all $t\in (0, 1]$}.
$$

Suppose on the contrary that $X\in \mathfrak{A}$ is a solution of $P_{\mathfrak{A}}=(P_{\mathfrak{A}}+Q_{\mathfrak{A}})^{1/2}X$. Write
$$X=X(t)=\left(\begin{matrix}x_{11}(t)&x_{12}(t)\\x_{21}(t)&x_{22}(t)\end{matrix}\right),$$ where $x_{ij}\in C[0,1]$, $i,j=1,2$ with $x_{12}(0)=x_{21}(0)=x_{12}(1)=x_{21}(1)=0$. Then

\begin{equation}\label{formula_for_X}
X(t)=(P(t)+Q(t))^{-1/2}P(t)=\frac12\left(\begin{matrix}\sqrt{1+c_t}+\sqrt{1-c_t}&0\\-\sqrt{1+c_t}+\sqrt{1-c_t}&0\end{matrix}\right),\quad {\rm ~for~all~} t\in (0,1].
\end{equation}
It follows that
$$0=X_{21}(0)=\lim_{t\to 0}x_{21}(t)=\lim_{t\to 0} \frac12\big(-\sqrt{1+c_t}+\sqrt{1-c_t}\big)=-\frac{1}{\sqrt{2}},$$
which is a contradiction.
\end{proof}

\begin{theorem} Let $\mathscr{H}$ be any Hilbert $C^*$-module and $P,Q\in\mathcal{L}(\mathscr{H})$ be two projections. Then for any $\varepsilon\in (0,1)$, there exists a projection $Q'\in\mathcal{L}(\mathscr{H})$ such that $\|Q-Q'\|<\varepsilon$ and the equation $(P+Q')^{1/2}X=P,\ X\in\mathcal{L}(\mathscr{H})$ has a solution.

\end{theorem}
\begin{proof}
It is known that the $C^*$-algebra $\mathfrak{A}$ defined by \eqref{A} is the universal unital $C^*$-algebra generated by two projections \cite{Raeburn-Sinclair}. By the universality of $\mathfrak{A}$, given two projections $P$ and $Q$ in $\mathcal{L}(\mathscr{H})$, we get a $*$-homomorphism $\psi:\mathfrak{A}\to\mathcal{L}(\mathscr{H})$ such that $\psi(P_{\mathfrak{A}})=P$ and $\psi(Q_{\mathfrak{A}})=Q$.

Let $h:[0,1]\to[\varepsilon,1]$ be a linear homeomorphism, and let $\mu:\mathfrak{A}\to \mathfrak{A}$ be a $*$-homomorphism defined by
$$
\mu(f)(t)=\left\lbrace\begin{array}{cl}f(0)&\mbox{if\ }t\in[0,\varepsilon];\\f(h^{-1}t)&\mbox{if\ }t\in[\varepsilon,1],\end{array}\right.\qquad f\in C\big([0,1],M_2(\mathbb{C})\big).
$$
Set $Q'_A=\mu(Q_{\mathfrak{A}})$. Then $Q'_A$ is a projection, and $\lim_{\varepsilon\to 0}\|Q_{\mathfrak{A}}-Q'_A\|=0$. Set $Q'=\psi(Q'_A)$, then $\|Q-Q'\|\leq\|Q_{\mathfrak{A}}-Q'_A\|$.

For $t\in(\varepsilon,1]$, the equation $P_{\mathfrak{A}}(t)=(P_{\mathfrak{A}}(t)+Q'_A(t))^{1/2}X(t)$ has a unique solution $X_A(t)$ as formulated by \eqref{formula_for_X}, with $\lim_{t\to\varepsilon}x_{21}(t)=-\frac{1}{\sqrt{2}}$. Note that for $t\in[0,\varepsilon]$ we have $Q'_A(t)=P_{\mathfrak{A}}(t)=\left(\begin{smallmatrix}1&0\\0&0\end{smallmatrix}\right)$, hence for $t\in[0,\varepsilon]$ we can take
$$
X_A(t)=\left(\begin{matrix}\frac{1}{\sqrt{2}}&0\\-\frac{t}{\varepsilon\sqrt{2}}&0\end{matrix}\right)
$$
as a solution for $P_{\mathfrak{A}}(t)=(P_{\mathfrak{A}}(t)+Q'_A(t))^{1/2}X(t)$, and, as $x_{21}$ is continuous and $x_{21}(0)=0$, we have $X_A\in \mathfrak{A}$. Then $X=\psi(X_A)\in\mathcal{L}(\mathscr{H})$ is a solution for $P=(P+Q')^{1/2}X$.
\end{proof}

\vspace{2ex}
\noindent\textbf{Acknowledgments}

The authors would like to thank Dr. Haiyan Zhang for the counterexample given in Remark~\ref{rem:Zhangs observation}.

\vspace{2ex}

\bibliographystyle{amsplain}

\begin{thebibliography}{99}

\bibitem{Anderson-Duffin} W. N. Anderson, Jr. and R. J. Duffin, \textit{Series and parallel addition of matrices}, J. Math. Anal. Appl. \textbf{26} (1969), 576--594.

\bibitem{ACG2} M. L. Arias, G. Corach, and M. C. Gonzalez, \textit{Generalized inverses and Douglas equations}, Proc. Amer. Math. Soc. \textbf{136} (2008), no. 9, 3177--3183.

\bibitem{Dajic-Koliha} A. Daji\'{c} and J. J. Koliha, \textit{Positive solutions to the equations $AX=C$ and $XB=D$ for Hilbert space operators}, J. Math. Anal. Appl. \textbf{333} (2007), 567--576.

\bibitem{DOU} R. G. Douglas, \textit{On majorization, factorization, and range inclusion of operators on Hilbert space}, Proc. Amer. Math. Soc. \textbf{17} (1966), 413--415.

\bibitem{Fang-Moslehian-Xu} X. Fang, M. S. Moslehian and Q. Xu, \textit{On majorization and range inclusion of operators on Hilbert $C^*$-modules}, Linear Multilinear Algebra \textbf{66}, no. 12, 2493--2500.

\bibitem{Fang-Yu-Yao}X. Fang, J, Yu and H. Yao, \textit{Solutions to operator equations on Hilbert $C^*$-modules}, Linear Algebra Appl. \textbf{431} (2009), 2142--2153.

\bibitem{FS} L. A. Fialkow and H. Salas, \textit{Majorization, factorization and systems of linear operator equations}, Math. Balkanica (N.S.) \textbf{4} (1990), no. 1, 22--34.

\bibitem{Fillmore-Williams}P. A. Fillmore and J. P. Williams, \textit{On operator ranges}, Adv. Math. \textbf{7} (1971), 254--281.

\bibitem{KM} C. G. Khatri and S. K. Mitra, \textit{Hermitian and nonnegative definite solutions of linear matrix equations}, SIAM J. Appl. Math. \textbf{31} (1976), no. 4, 579--585.

\bibitem{Lance}E. C. Lance, \textit{Hilbert $C^*$-modules-A toolkit for operator algebraists}, Cambridge University Press, Cambridge, 1995.

\bibitem{Liang-Deng} W. Liang and C. Deng, \textit{The solutions to some operator equations with corresponding operators not necessarily having closed ranges}, Linear Multilinear Algebra (to appear), doi:10.1080/03081087.2018.1464548

\bibitem{Liu-Luo-Xu}N. Liu, W. Luo and Q. Xu, \textit{The polar decomposition for adjointable operators on Hilbert C*-modules and centered operators}, Adv. Oper. Theory \textbf{3} (2018), no. 4, 855--867.
\bibitem{MT} V. M. Manuilov and E. V. Troitsky, \textit{Hilbert $C^*$-modules}, Translated from the 2001 Russian original by the authors, Translations of Mathematical Monographs, 226. American Mathematical Society, Providence, RI, 2005.

\bibitem{MEMM} Z. Mousavi, R. Eskandari, M. S. Moslehian, and F. Mirzapour, \textit{Operator equations $AX+YB=C$ and $AXA^*+BYB^*=C$ in Hilbert $C^*$-modules}, Linear Algebra Appl. \textbf{517} (2017), 85--98.

\bibitem{MKX} M. S. Moslehian, M. Kian, and Q. Xu, \textit{Positivity of $2\times 2$ block matrices of operators}, Banach J. Math. Anal. (2019), doi:10.1215/17358787-2019-0019.


\bibitem{NAK} R. Nakamoto, \textit{On the operator equation $THT = K$}, Math. Japon. \textbf{18} (1973), 251--252.

\bibitem{PS} D. Popovici and Z. Sebestyén, \textit{Factorizations of linear relations}, Adv. Math. \textbf{233} (2013), 40--55.

\bibitem{RM} C. R. Rao, and S. K. Mitra, \textit{Generalized inverse of matrices and its applications}, John Wiley \& Sons, Inc., New York-London-Sydney, 1971.

\bibitem{Raeburn-Sinclair} I. Raeburn and A. M. Sinclair, \textit{The $C^*$-algebra generated by two projections}, Math. Scand. \textbf{65} (1989), 278--290.

\bibitem{WW} Q. Wang and Z. Wu, \textit{Common Hermitian solutions to some operator equations on Hilbert $C^*$-modules}, Linear Algebra Appl. \textbf{432} (2010), no. 12, 3159--3171.

\bibitem{Wegge-Olsen}N. E. Wegge-Olsen, \textit{$K$-theory and $C^*$-algebras: A friendly approach}, Oxford Univ. Press, Oxford, England, 1993.

\bibitem{Xu-Sheng} Q. Xu and L. Sheng, \textit{Positive semi-definite matrices of adjointable operators on Hilbert $C^*$-modules}, Linear Algebra Appl. \textbf{428} (2008), 992--1000.



\end{thebibliography}

\end{document}